\documentclass[12pt]{amsart}

\usepackage{amsthm,amsfonts,amssymb,graphicx,cite,mathrsfs,enumitem}

\allowdisplaybreaks \allowdisplaybreaks[4]
\newtheorem{mytheorem}{Theorem}

\newtheorem {example} {Example}

\newtheorem{lemma}{Lemma}[section]
\newtheorem{theorem}{Theorem}[section]
\newtheorem{remark}{Remark}[section]

\newcounter{saveeqn}%

\begin{document}
	\title[Jacobian conjecture in $\mathbb{R}^n$]
	{Global injectivity and real Jacobian conjecture \\
    in $\mathbb{R}^n$}
	
	\author[Hu, Llibre, Tang]{Xinhao Hu$^a$,  Jaume Llibre$^b$ and Yilei Tang$^a$}

	\address{$^a$ School of Mathematical Sciences, Shanghai Jiao Tong University, Shanghai, 200240, P.R. China} 
    \email{huxinhao@sjtu.edu.cn (X. Hu), mathtyl@sjtu.edu.cn (Y. Tang)}

	\address{$^b$ Departament de Matem\`{a}tiques, Universitat Aut\`{o}noma de Barcelona, 08193 Bellaterra, Barcelona, Catalonia, Spain \newline Reial Acad\`emia de Ci\`encies i Arts de Barcelona, La Rambla 115, 08002, Barcelona, Catalonia, Spain}
	\email{jaume.llibre@uab.cat (J. Llibre)}

		\begin{abstract}
		In this paper, we study the global injectivity of local diffeomorphisms in $\mathbb{R}^n$, extending the criterions of Braun and Llibre in [{\it Pacific J. Math.} {\bf 333} (2024) 1--15] from the planar radial case to arbitrary dimension. Then we extend the results of Llibre and Mello in [{\it Bull. Sci. Math.} {\bf 213} (2026) 103889] to arbitrary dimension, obtaining two necessary and sufficient conditions for the real Jacobian conjecture in $\mathbb{R}^n$ by qualitative theory of differential equations. Finally, using weighted Poincar\'e compactification in $\mathbb{R}^n$, we provide a sufficient condition for the real Jacobian conjecture.
	\end{abstract}
	
	
	\date{}
	\dedicatory{}
	
	\maketitle
	\section{Introduction and main results}
		\setcounter{equation}{0}
	Let $U$ and $V$ be open subsets of $\mathbb R^n$.  An important problem is to determine when a local diffeomorphism
	\[
	f:U\to V
	\]
	is in fact globally injective, equivalently, a global diffeomorphism onto
	its image.  The inverse function theorem gives local invertibility in a neighborhood near the point $x$ under the
	condition 
	\[
	J_f(x):=\det Df(x)\neq 0
	\]
	where $Df(x)$ denotes the Jacobian matrix of $f(x)$. This gives invertibility
	in a neighborhood of the point $x$, but the condition alone does not ensure
	the global injectivity of the map.
	Thus one natural question is to ask which additional geometric, topological, or
	dynamical assumptions can promote local invertibility to global injectivity (see, for example in \cite{Parthasarathy83,Xavier}). Even for polynomial maps, the problem remains far from trivial and it brings to the celebrated \textit{Jacobian conjecture} and \textit{real Jacobian conjecture}, which will be introduced in detail later on.
	
Recently, a new approach to this question in dimension two has been developed by Braun and Llibre \cite{Braun} using planar {\it parallel} vector fields (see \cite{Li,Neumann} for a precise definition). In short, a planar vector field is called parallel if it is topologically
	equivalent to one of the three {\it canonical vector fields}
	$(-x_2,x_1)$ in $\mathbb{R}^2\setminus\{0\}$, $(x_1,x_2)$ in $\mathbb{R}^2\setminus\{0\}$, and $(1,0)$ in $\mathbb{R}^2$,
	corresponding respectively to the {\it annular}, {\it radial}, and {\it strip} cases. Given a
	local diffeomorphism $f:U\to\mathbb R^2$ and a vector field $\mathcal X$ defined on
	$f(U)$, where $U$ is an open subset of $\mathbb R^2$, they considered the
	pullback vector field $Df(x)^{-1}\mathcal X(f(x))$ for $x\in U$.
One of the results in \cite{Braun} is stated as follows.
	\begin{theorem}[see \cite{Braun}]
		\label{Th BL}
	Let $U$ be a simply connected open subset of $\mathbb{R}^2$. Let
	$f=(f_1,f_2):U\rightarrow \mathbb{R}^2$ be a $\mathcal{C}^2$ local
	diffeomorphism, and let $\mathcal{X}:f(U)\rightarrow\mathbb{R}^2$ be a
	$\mathcal{C}^1$ vector field topologically equivalent to one of the canonical
	annular, radial, or strip vector fields, surrounding a singularity in the
	first two cases. Then the following two conditions are equivalent:
		\begin{enumerate}[leftmargin=*]
		\item[\rm(i)] The pullback vector field $Df(x)^{-1}\mathcal X(f(x))$ is topologically equivalent to the vector field $\mathcal{X}$.
		\item[\rm(ii)] The map $f$ is a global diffeomorphism onto $f(U)$.
	\end{enumerate}
	\end{theorem}
 Let $U$ and $V$ be open subsets of $\mathbb R^n$. We say that two vector fields
$\mathcal Z:U\to\mathbb R^n$ and $\mathcal W:V\to\mathbb R^n$ are 
\emph{topologically equivalent} if there exists a homeomorphism
$h:U\to V$ carrying orbits of $\mathcal Z$ onto orbits of $\mathcal W$ and
preserving the orientation of the orbits.
	 Theorem \ref{Th BL} unifies
		several previous criteria for global injectivity. For example, it recovers
		the planar form of the Hadamard global invertibility theorem (see, for example, \cite{Deimling}). It also contains the dynamical results of the planar real
		Jacobian problem due to Gavrilov \cite{Gavrilov97} and Sabatini \cite{Sabatini}. Indeed, the pullback of canonical annular vector fields leads to a planar Hamiltonian vector field with Hamiltonian $H=(f_1^2+f_2^2)/2$, and a polynomial local diffeomorphism $f=(f_1,f_2)$ of the plane is injective precisely when the corresponding Hamiltonian vector field has a global center. More recently, Domingues, Llibre and Mello \cite{DominguesLlibreMello26} used the strip case in their Theorem~1 to obtain a sufficient condition for the real Jacobian conjecture in $\mathbb R^2$ using the topological structure of the polynomial Hamiltonian vector field $(\partial_{x_2}f_2,-\partial_{x_1}f_2)$, which is exactly the pullback of the canonical strip vector field. Moreover,
		Llibre and Mello \cite{LlibreMello} used the radial case in their Theorem~1 to obtain a necessary and
		sufficient condition for the real Jacobian conjecture in $\mathbb R^2$ in
		terms of the boundedness of the polynomial vector field $(f_1 \partial_{x_2} f_2-f_2 \partial_{x_2} f_1,\,-f_1 \partial_{x_1} f_2+f_2 \partial_{x_1} f_1)$,
		which is exactly the pullback of the canonical radial vector field.
		
		One purpose of this paper is to extend the radial case of Theorem~1 in \cite{Braun} from the plane to $\mathbb R^n$.  The {\it radial} case has
	a natural higher-dimensional form, namely the  canonical radial vector field in $\mathbb{R}^n$:
	\begin{equation}\label{canonical-radial}
	\mathcal X_r(x)=(x_1,\ldots,x_n),\qquad x\in\mathbb R^n.
\end{equation}
Thus it is possible to formulate and prove a higher-dimensional criterion for global injectivity in terms of the pullback of a
	radial vector field (i.e., a vector field that is topologically equivalent to \eqref{canonical-radial}).  However, the annular and strip cases involve
	additional difficulties in higher dimensions.  The planar annular case is  related to closed orbits and global centers, while the
	strip case is closely related to the structure of one-dimensional foliations
	and the connectedness of level curves.  These are essential
	two-dimensional features, and their direct higher-dimensional analogues leave some gaps.  We shall return to this point in Section \ref{S4}.
	
	We now introduce some definitions and notations used in this paper.
	Let $U\subset\mathbb R^n$ be an open set. For $x_0\in U$, we say that a
	vector field $\mathcal X:U\to\mathbb R^n$ is \emph{radial surrounding
	$x_0$} if $\mathcal X|_{U\setminus\{x_0\}}$ is topologically equivalent to
	$\mathcal X_r$ on $\mathbb R^n\setminus\{0\}$.
	Let $f=(f_1,\ldots,f_n):U\to\mathbb R^n$ be a $\mathcal{C}^k$ local diffeomorphism,
	$k\geq 1$, and let $\mathcal X: f(U)\to\mathbb R^n$ be a $\mathcal{C}^1$ vector field.
	The {\it pullback of $\mathcal X$ by $f$} is defined by	\begin{equation}\label{pullback}
		\mathcal Y_{\mathcal X}(x)
		=
		Df(x)^{-1}\mathcal X(f(x)),\quad x\in\mathbb R^n.
	\end{equation}
	In particular, the pullback of the canonical radial vector field
	\eqref{canonical-radial} is
	\begin{equation}\label{radial-pullback}
		\mathcal Y_r(x)=Df(x)^{-1}f(x).
	\end{equation}
	In Theorem \ref{Th1} and its proof, $\mathcal Y_r$ denotes the pullback
	$\mathcal Y_{\mathcal X}$ of the radial vector field $\mathcal X$.
Then one of our main results is stated as follows.
	\begin{mytheorem}\label{Th1}
		Let $U$ be a simply connected open subset of $\mathbb R^n$.  Let
		$f=(f_1,\ldots,f_n):U\to\mathbb R^n$ be a $\mathcal{C}^2$ local diffeomorphism, and
		let $\mathcal X:f(U)\to\mathbb R^n$ be a $\mathcal{C}^1$ vector field which is radial
		surrounding a singularity.  Then the following two statements are equivalent:
		\begin{enumerate}[leftmargin=*]
			\item[\rm(i)] The pullback vector field $\mathcal Y_r$ is
			topologically equivalent to $\mathcal X$.
			\item[\rm(ii)] The map $f$ is a global diffeomorphism onto $f(U)$.
		\end{enumerate}
	\end{mytheorem}
	
	Theorem \ref{Th1} can be regarded as a higher-dimensional radial case of
	Theorem~1 of \cite{LlibreMello}.  It transforms a question about
	the global injectivity of a map into a question about the global dynamics of a vector field in $\mathbb{R}^n$.  This dynamical viewpoint is particularly well suited to polynomial maps and leads naturally to the  Jacobian conjecture  and real Jacobian conjecture.
	
In 1939, Keller \cite{Keller} considered polynomial maps in
$\mathbb{C}^n$ and proposed the following conjecture: if
$f=(f_1,\ldots,f_n):\mathbb{C}^n\rightarrow \mathbb{C}^n$ is a polynomial map
such that $\operatorname{det}Df(x)$ is a nonzero constant for every
$x\in\mathbb{C}^n$, then $f$ is injective. This is the well-known Jacobian conjecture. Note that if the polynomial map $f$ is injective, then it is a global diffeomorphism (see \cite{Bialynicki-Birula}). 
Recently, Alp\"oge announced a counterexample to the Jacobian
conjecture in dimension three. Consequently, the conjecture fails in
every dimension $n>2$, while the two-dimensional case remains open;
see Gao \cite{Gao26} for the summary of recent work.

In 1983, Randall \cite{Randall} considered the polynomial map $f$ in the Euclidean space
$\mathbb{R}^n$ and formulated the following conjecture: if
\begin{equation}
	\label{polynomial map}
	f=(f_1,\ldots,f_n):\mathbb R^n\to\mathbb R^n
\end{equation}
 is a polynomial map
satisfying $J_f(x)\neq0$ for $x\in\mathbb R^n$, then $f$ is injective. This conjecture is known as
the real Jacobian conjecture. The conjecture is false in full generality because of Pinchuk's counterexample in dimension two (see \cite{Pinchuk}). Nevertheless, it remains important to find additional
	effective conditions under which the conclusion is true.  Many such
	conditions have been obtained in dimension two (see \cite{Braundos,BraunGineLlibre16,DominguesLlibreMello26,GineLlibre21,TianCen24}), some of which were mentioned above. Specifically, using Theorem \ref{Th BL} in the annular case, or equivalently the result of
Gavrilov \cite{Gavrilov97} or Sabatini \cite{Sabatini}, Braun, Gin\'e and Llibre \cite{BraunGineLlibre16} obtained the following result. Here $\overline{p}$ denotes the highest homogeneous part of the polynomial $p$.
    \begin{theorem}[see \cite{BraunGineLlibre16}, Theorem~1]
        Let $f=(f_1, f_2): \mathbb{R}^2 \rightarrow \mathbb{R}^2$ be a polynomial map such that $J_f$ is nowhere zero and $f(0,0)=(0,0)$. If  $\overline{f_1\partial_{x_1}f_1+f_2\partial_{x_1}f_2}$ and $\overline{f_1\partial_{x_2}f_1+f_2\partial_{x_2}f_2}$ have no real linear factors in common, then $f$ is injective.
    \end{theorem}
Using Theorem \ref{Th BL} in the strip case, Domingues, Llibre and Mello
\cite{DominguesLlibreMello26} proved the following theorem.
\begin{theorem}[see \cite{DominguesLlibreMello26}, Theorem~1]
    Let $f=(f_1, f_2): \mathbb{R}^2 \rightarrow \mathbb{R}^2$ be a polynomial map such that $J_f$ is nowhere zero. If $\overline{\partial_{x_1}f_2}$ and $\overline{\partial_{x_2}f_2}$ have no real linear factors in common, then $f$ is injective.
\end{theorem}
Using Theorem \ref{Th BL} in the radial case, Llibre and Mello
\cite{LlibreMello} obtained the following two results.
\begin{theorem}[see \cite{LlibreMello}, Theorem~1]
\label{Th-L-M1}
    Let $f=(f_1, f_2): \mathbb{R}^2 \rightarrow \mathbb{R}^2$ be a polynomial local orientation reversing diffeomorphism and assume $f(0,0)=(0,0)$. Define the polynomial vector field
\begin{equation}
\label{vector G}
    \mathcal{G}: \mathbb{R}^2 \rightarrow \mathbb{R}^2, \quad  \mathcal{G}=( \mathcal{G}_1, \mathcal{G}_2)=(f_1 \partial_{x_2} f_2-f_2 \partial_{x_2} f_1,\,-f_1 \partial_{x_1} f_2+f_2 \partial_{x_1} f_1).
\end{equation}
Then the vector field $ \mathcal{G}$ is bounded if and only if $f$ is injective.
\end{theorem}
    \begin{theorem}[see \cite{LlibreMello}, Theorem~2]
    \label{Th-L-M2}
        Let $f=(f_1, f_2): \mathbb{R}^2 \rightarrow \mathbb{R}^2$ be a polynomial local orientation reversing diffeomorphism. Assume  that $f(0,0)=(0,0),\ \operatorname{deg}(f_1)=\operatorname{deg}(f_2)>1$ and $\overline{f_1}$ is not a multiple of $\overline{f_2}$. If the components $ \mathcal{G}_1$ and $ \mathcal{G}_2$ of the vector field \eqref{vector G} satisfy
$$
\overline{ \mathcal{G}_1}(1, s)<0 \quad \text { and } \quad \overline{ \mathcal{G}_2}(0,1)<0, \quad \forall s \in \mathbb{R},
$$
then $f$ is injective.
    \end{theorem}

	    While the higher-dimensional case is much less understood (see, for example, \cite{CimaGasullManosas96,Liu}). Recent work of Liu and Tian \cite{Liu} gives several
		$n$-dimensional sufficient conditions in terms of the highest homogeneous or
		quasi-homogeneous parts of the polynomial map and of the gradient vector
		field associated with $\|f\|^2/2$.
		
		In the present paper we use Theorem \ref{Th1} to obtain two necessary and sufficient conditions in order that the $n$-dimensional real Jacobian conjecture holds by using the qualitative theory of $n$-dimensional differential systems. Let
		$f:\mathbb R^n\to\mathbb R^n$ be the polynomial map \eqref{polynomial map} such that $f(0)=0$. Because $\mathbb R^n$ is connected, $J_f$ has a constant
		sign. By composing $f$ with an orientation-reversing linear isomorphism if necessary, we may assume without loss of generality that $J_f(x)<0$ for all $x\in\mathbb R^n$.
		By \eqref{radial-pullback} and Cramer's rule, $J_f(x)\cdot\mathcal{Y}_r(x)$ is a polynomial vector field. Multiplying $\mathcal Y_r$ by the nowhere-vanishing function $J_f<0$ does not change its orbit foliation because, under the time scale $J_f$, it only reverses the time orientation of the orbits.
		It is also useful to introduce
	\begin{equation}\label{Hz}
		H(x)=\frac12\|f(x)\|^2
		=
		\frac12\sum_{i=1}^n f_i(x)^2.
	\end{equation}
		Then the {\it gradient vector field} with respect to $H$ is as follows:
	\[
	\nabla H(x)=Df(x)^T f(x),
	\]
	and the {\it negative gradient vector field} $-\nabla H$ gives another natural
	dynamical system associated with $f$. 
		We shall say that a vector field on $\mathbb R^n$ is \emph{bounded} if the positive orbit of every point eventually enters and remains in a compact subset of $\mathbb R^n$. We say that the origin is \emph{globally
			asymptotically stable} if it is an equilibrium and every positive
		orbit tends to the origin as $t\to+\infty$.
	
	Our second main result is stated as follows.
	
\begin{mytheorem}
	\label{Th2}
	Let $f: \mathbb{R}^n \rightarrow \mathbb{R}^n$ be a polynomial map such that $J_f(x)<0$ for all $ x\in\mathbb{R}^n$, and $f(0)=0$. Then the following statements are equivalent
	\begin{enumerate}[leftmargin=*]
		\item[\rm(i)] The polynomial vector field $J_f\cdot\mathcal{Y}_r$ (or $-\nabla H$) is bounded.
		\item[\rm(ii)] The origin of the polynomial vector field $J_f\cdot\mathcal{Y}_r$ (or $-\nabla H$) is globally asymptotically stable.
		\item[\rm(iii)] The map $f$ is injective.
	\end{enumerate}
\end{mytheorem}  

	Theorem \ref{Th2} is an $n$-dimensional generalization of Theorem \ref{Th-L-M1}.
	However, the proof here is based on the properties of limit sets of flows and index theory of $n$-dimensional differential equations without the special structures or properties in the planar differential equations. Moreover, the proof is based on the qualitative theory of differential systems rather than on purely algebraic or functional analytic proofs. However, the proof of \cite[Theorem 2]{LlibreMello} relies on the correctness of the Markus--Yamabe conjecture in dimension $2$, which is false in dimension $n>2$; thus we cannot extend the proof of \cite[Theorem 2]{LlibreMello}. We refer the reader to \cite{LlibreMello,MarkusYamabe} for this conjecture.
		
		Finally, we derive from Theorem \ref{Th2} an explicit sufficient condition for the $n$-dimensional real Jacobian conjecture. Indeed, we give a sufficient condition such that the polynomial vector field $J_f\cdot\mathcal{Y}_r$ is bounded and then by Theorem \ref{Th2} the map $f$ is injective.
    Before the statement of Theorem \ref{Th3}, we need the following definitions. Let
\[
\mathbf w=(w_1,\ldots,w_n)\in \mathbb N^n
\]
be a weight vector. For a monomial
$x^\alpha=x_1^{\alpha_1}\cdots x_n^{\alpha_n}$, we define its
weighted degree with respect to $\mathbf w$ by
\[
\deg_{\mathbf w} x^\alpha=w_1\alpha_1+\cdots+w_n\alpha_n.
\]
We define the weighted degree of a polynomial $p$ with respect to $\mathbf{w}$, denoted $\deg_{\mathbf{w}} p$, as the maximum of the weighted degrees of its monomials. Moreover, a polynomial $p$ is called $\mathbf{w}$-quasi-homogeneous if all its monomials have the same weighted degree with respect to $\mathbf{w}$. Equivalently, there exists an integer $d$, the weighted degree, such that
\[
p(t^{w_1}x_1,\dots,t^{w_n}x_n)=t^d p(x),\qquad t>0.
\]
    The following is our third main result.    
	\begin{mytheorem}
		\label{Th3}
		Let $f:\mathbb R^n\to\mathbb R^n$ be a polynomial map such that
	$J_f(x)<0$ for all $x\in\mathbb R^n$, and  $f(0)=0$.  Let 		$\mathbf w\in\mathbb N^n$
		be a weight vector. Suppose that
		\[
		\deg_{\mathbf w}f_1=\cdots=\deg_{\mathbf w}f_n=m,
		\]
			and let $F=(F_1,\ldots,F_n)$
			be the $\mathbf w$-quasi-homogeneous part of $f$, where each $F_i$ has weighted degree $m$. If
		\[
		J_F(x)<0,
		\qquad x\in\mathbb R^n\setminus\{0\},
		\]
		then $f$ is injective.
	\end{mytheorem}

Indeed, Theorem \ref{Th3} can be seen as a generalization of Theorem \ref{Th-L-M2} in some aspects, although the forms of conditions differ slightly. Here we give a sufficient condition in arbitrary dimension and use the weighted Poincar\'e compactification instead of the usual compactification.
Unlike some previous methods in higher dimensions (see, for example, \cite{CimaGasullManosas96,Liu}), we analyze directly the
		polynomial vector field $J_f\cdot \mathcal Y_r$
		by the weighted Poincar\'e compactification in $\mathbb R^n$.  The highest
	$\mathbf w$-quasi-homogeneous part of $J_f\mathcal Y_r$ is controlled by $J_F$, and the condition
	$J_F<0$ makes the weighted boundary at infinity repelling. It follows that all positive orbits of
	$J_f\cdot \mathcal Y_r$ are bounded.  Theorem \ref{Th2} then yields the injectivity of
	$f$. In this way, Theorem \ref{Th3} provides a sufficient dynamical condition for the real
	Jacobian conjecture.
	
The paper is organized as follows. In Section 2 we collect some preliminary
results on pullback vector fields, the  index of isolated singularities, and the weighted
Poincar\'e compactification of polynomial vector fields in $\mathbb{R}^n$. In
Section 3 we prove the main results. We first establish the higher-dimensional
radial criterion of Theorem \ref{Th BL} for global injectivity and then apply it to
polynomial maps with nonvanishing Jacobian determinant to obtain two necessary
and sufficient dynamical conditions for injectivity. We also prove an explicit
sufficient condition in terms of the highest quasi-homogeneous part of the polynomial
map by analyzing the compactified vector field at weighted infinity. Finally, in
Section 4 we discuss the possible extension of the strip case to higher dimensions and point out that the additional connectedness assumption is needed.

	\bigskip
	\section{Preliminaries}
\setcounter{equation}{0}

	\subsection{Some properties of pullback vector fields}
	The next result is the higher-dimensional version of \cite[Lemma 9]{Braun}. For completeness, we provide a brief proof below.
	\begin{lemma}
		\label{lem:pullback-basic}
		Let $U,V\subset\mathbb{R}^n$ be open sets. Let
		$f:U\to\mathbb{R}^n$ be a $\mathcal{C}^2$ local diffeomorphism with $f(U)\subset V$ and $\mathcal{X}:V\to\mathbb{R}^n$ be a $\mathcal{C}^1$ vector field.  Then:		
		\begin{enumerate}[leftmargin=*]
			\item[\rm(a)] A point $x\in U$ is a singularity of $\mathcal{Y}_{\mathcal{X}}$ defined by \eqref{pullback}
			if and only if $f(x)$ is a singularity of $\mathcal{X}$.		
			\item[\rm(b)]  If $\gamma$ is an orbit of $\mathcal{Y}_{\mathcal{X}}$, then
			\begin{equation}\label{eq:orbit-image}
				(f\circ\gamma)'(t)=\mathcal{X}(f(\gamma(t)))
			\end{equation}
			for every $t$ in the maximal interval of definition of $\gamma$. In particular,
			$f\circ\gamma$ is contained in an orbit of $\mathcal{X}$.			
			
			\item[\rm(c)] If $G:V\to\mathbb{R}$ is a first integral of $\mathcal{X}$, then
			$H=G\circ f$ is a first integral of $\mathcal{Y}_{\mathcal{X}}$.
		\end{enumerate}
	\end{lemma}
	
	\begin{proof}
		Statement (a) follows immediately from \eqref{pullback}, since $Df(x)$ is
		invertible for every $x\in U$.	
		To prove (b), let $\gamma$ be an orbit of $\mathcal{Y}_{\mathcal{X}}$. Then
		\[
		(f\circ\gamma)'(t)
		=
		Df(\gamma(t))\gamma'(t)
		=
		Df(\gamma(t))\mathcal{Y}_{\mathcal{X}}(\gamma(t))
		=
		\mathcal{X}(f(\gamma(t))),
		\]
		which is exactly \eqref{eq:orbit-image}. Hence $f\circ\gamma$ is an integral
		curve of $\mathcal{X}$, and therefore it is contained in an orbit of
		$\mathcal{X}$.		
		Finally for statement (c), let $G$ be a first integral of $\mathcal{X}$ and set $H=G\circ f$.
		By \eqref{eq:orbit-image}, the map $f$ sends each orbit of
		$\mathcal{Y}_{\mathcal{X}}$ into an orbit of $\mathcal{X}$. Since $G$ is
		constant on the latter, it follows that $H$ is constant on the former, and
		therefore $H$ is a first integral of $\mathcal{Y}_{\mathcal{X}}$.
	\end{proof}

\subsection{Index of singularity of polynomial vector field}
Let $q\in\mathbb{R}^n$ be an isolated singularity of a vector field $\mathscr X$, and let $U_q$ be a neighborhood of $q$ such that $\mathscr X(x)\neq0$ for all $x\in U_q\setminus\{q\}$. Define the mapping
\[
\varphi:\partial U_q\longrightarrow\mathbb{S}^{n-1},
\qquad
x\longmapsto\frac{\mathscr X(x)}{\|\mathscr X(x)\|}.
\]
The index of the vector field $\mathscr X$ at $q$, denoted by $i_{\mathscr X}(q)$, is defined as the Brouwer degree of the map $\varphi$; see \cite{Browder,Deimling} for further details. Roughly speaking, the Brouwer degree is an integer-valued invariant associated with a triple $(f,U,q)$, where $U\subset\mathbb{R}^n$ is a bounded open set, $f:\bar U\to\mathbb{R}^n$ is continuous and $q\notin f(\partial U)$. It is denoted by $d_B(f,U,q)$ and is characterized by the four properties: existence, additivity, homotopy invariance and local constancy.
	
	The following result was given in \cite[Theorem 2]{Barreira}.
\begin{theorem}[see \cite{Barreira}]
	\label{Th index}
 Assume that the polynomial vector field $\mathscr{X}$ in $\mathbb{R}^n$ only has finitely many finite singularities $q_1, \ldots, q_l$. If the infinity of $\mathscr{X}$ (i.e. the boundary $\mathbb{S}^{n-1}$ of $\mathbb{D}^n$) is collapsed to a point, which is a repeller, then
 $$
 \sum_{j=1}^l i_\mathscr{X}\left(q_j\right)=(-1)^n .
 $$
\end{theorem}
	
The following result can be seen in  \cite[p. 37]{Milnor}.

\begin{lemma}[see \cite{Milnor}]
	\label{Lemma index}
	 Let $q \in \mathbb{R}^n$ be a singularity of polynomial vector field $\mathscr{X}$ with $\operatorname{det} D \mathscr{X}(q) \neq 0$. Then the index of $\mathscr{X}$ at singularity $q$ is $i_\mathscr{X}(q)=\operatorname{sign} \operatorname{det} D \mathscr{X}(q)$.
\end{lemma}

\subsection{Weighted Poincar\'e compactification of polynomial vector fields in $\mathbb R^n$}
Recall that the compactification of polynomial vector fields was introduced by
Poincar\'e in \cite{Poincare} for planar vector fields and was later extended
to higher dimensions in \cite{Cima}. We refer the reader to Chapter 5 of
\cite{Dumortier} or to \cite{Cima} for a detailed exposition.

In this subsection we recall a weighted version of the Poincar\'e compactification for
polynomial vector fields in $\mathbb R^n$. For more details, one can refer to \cite{BraunValls,Garcia2006,Perez2024,Valls2024}. We only provide a brief introduction below. Let
\[
\mathbf w=(w_1,\ldots,w_n)\in \mathbb N^n
\]
be a weight vector and let
\begin{equation}
	\label{poly vec}	
	\mathscr X=(P_1,\ldots,P_n)
\end{equation}
be a polynomial vector field in $\mathbb R^n$. 
We say that $\mathscr X$ has weighted degree $d$ with respect to
$\mathbf w$ if
\[
\deg_{\mathbf w} P_j\le d+w_j-1,\qquad j=1,\ldots,n,
\]
and at least one equality holds. The $\mathbf w$-quasi-homogeneous
part of $\mathscr X$ of degree $d$ is denoted by
\[
\mathscr X_d=(P_{1,d},\ldots,P_{n,d}),
\]
where $P_{j,d}$ is the sum of all monomials of $P_j$ of 
weighted degree $d+w_j-1$ with respect to $\mathbf w$ such that
\[
P_{j,d}(\lambda^{w_1}x_1,\ldots,\lambda^{w_n}x_n)
=
\lambda^{d+w_j-1}P_{j,d}(x),
\qquad \lambda>0 .
\]

We now describe the weighted Poincar\'e compactification in local charts.
Fix $k\in\{1,\ldots,n\}$ and put
\[
\sigma(i)=
\begin{cases}
	i, & i<k,\\
	i+1, & i\ge k.
\end{cases}
\]
In the chart $U_k$ we use the local variables $z=(z_1,\ldots,z_{n-1},z_n)$
given by the weighted change of variables
\[
x_k=\frac{1}{z_n^{w_k}},
\qquad
x_{\sigma(i)}=\frac{z_i}{z_n^{w_{\sigma(i)}}},
\qquad i=1,\ldots,n-1.
\]
The hyperplane $z_n=0$ corresponds to the part of
weighted infinity of $\mathscr{X}$. We also set
\[
\tilde z^{(k)}
=
(z_1,\ldots,z_{k-1},1,z_k,\ldots,z_{n-1}).
\]
Define
\[
\widetilde P_j(z)
=
z_n^{d+w_j-1}
P_j
\left(
\frac{z_1}{z_n^{w_1}},\ldots,\frac{z_{k-1}}{z_n^{w_{k-1}}},
\frac{1}{z_n^{w_k}},
\frac{z_k}{z_n^{w_{k+1}}},\ldots,\frac{z_{n-1}}{z_n^{w_n}}
\right).
\]
Then $\widetilde P_j$ is a polynomial in $(z_1,\ldots,z_n)$, and
\[
\widetilde P_j(z_1,\ldots,z_{n-1},0)
=
P_{j,d}(\tilde z^{(k)}).
\]
After the time rescaling $dt=z_n^{d-1}\,d\tau$,
the weighted compactified vector field is given in the chart $U_k$ by
\begin{equation}
	\label{Uk sys}
	\begin{aligned}
		\frac{dz_i}{d\tau}
		&=
		\widetilde P_{\sigma(i)}(z)
		-
		\frac{w_{\sigma(i)}}{w_k}
		z_i\widetilde P_k(z),
		\qquad i=1,\ldots,n-1,\\
		\frac{dz_n}{d\tau}
		&=
		-\frac{1}{w_k}
		z_n\widetilde P_k(z).
	\end{aligned}
\end{equation}
Thus $z_n=0$ is invariant, and the restriction of the compactified vector
field to $z_n=0$ describes the part of  dynamics of $\mathscr X$ at weighted
infinity.

Similarly, in the chart $V_k$, the weighted change of variables is
\[
x_k=-\frac{1}{z_n^{w_k}},
\qquad
x_{\sigma(i)}=\frac{z_i}{z_n^{w_{\sigma(i)}}},
\qquad i=1,\ldots,n-1.
\]
In this case one sets
\[
\hat z^{(k)}
=
(z_1,\ldots,z_{k-1},-1,z_k,\ldots,z_{n-1}),
\]
and defines
\[
\widehat P_j(z)
=
z_n^{d+w_j-1}
P_j
\left(
\frac{z_1}{z_n^{w_1}},\ldots,\frac{z_{k-1}}{z_n^{w_{k-1}}},
-\frac{1}{z_n^{w_k}},
\frac{z_k}{z_n^{w_{k+1}}},\ldots,\frac{z_{n-1}}{z_n^{w_n}}
\right).
\]
Then
\[
\widehat P_j(z_1,\ldots,z_{n-1},0)
=
P_{j,d}(\hat z^{(k)}),
\]
and the compactified vector field in $V_k$ is
\begin{equation}
	\label{Vk sys}
	\begin{aligned}
		\frac{dz_i}{d\tau}
		&=
		\widehat P_{\sigma(i)}(z)
		+
		\frac{w_{\sigma(i)}}{w_k}
		z_i\widehat P_k(z),
		\qquad i=1,\ldots,n-1,\\
		\frac{dz_n}{d\tau}
		&=
		\frac{1}{w_k}
		z_n\widehat P_k(z).
	\end{aligned}
\end{equation}
When $w_1=\cdots=w_n=1$, this construction reduces to the usual
Poincar\'e compactification.

\bigskip
	\section{Proof of main results}
	\setcounter{equation}{0}

	\subsection{Proof of Theorem \ref{Th1}}
	The proof that {\rm(ii)} implies {\rm(i)} is immediate from the definition of pullback, since $f$ itself provides the conjugacy. We only need to prove that {\rm(i)} implies {\rm(ii)}. Let $p\in f(U)$ be the singularity of $\mathcal X$. By Lemma \ref{lem:pullback-basic}(a) there exists a singularity $x_0\in U$ of $\mathcal Y_r$ such that $f(x_0)=p$. Since $\mathcal X$ is radial surrounding a singularity, it has a unique singularity, and by the pullback property the same holds for $\mathcal Y_r$, so that $x_0$ is the unique singularity of $\mathcal Y_r$.

	We first show that distinct nonsingular orbits of $\mathcal Y_r$ are mapped into distinct orbits of $\mathcal X$. Let $\gamma_1$ and $\gamma_2$ be two distinct nonsingular orbits of $\mathcal Y_r$. Since $\mathcal Y_r$ is radial surrounding $x_0$, both $\gamma_1$ and $\gamma_2$ have $x_0$ as their $\alpha$-limit set, and therefore they intersect arbitrarily small neighborhoods of $x_0$. Take a neighborhood $U_{x_0}$ of $x_0$ such that $f|_{U_{x_0}}:U_{x_0}\to f(U_{x_0})$ is a diffeomorphism. Then $\gamma_1\cap U_{x_0}$ and $\gamma_2\cap U_{x_0}$ lie in distinct local orbits of $\mathcal Y_r$ near $x_0$, and since $f|_{U_{x_0}}$ is a local conjugacy between $\mathcal Y_r$ and $\mathcal X$, their images $f(\gamma_1\cap U_{x_0})$ and $f(\gamma_2\cap U_{x_0})$ lie in distinct local orbits of $\mathcal X$ near $p$. As $\mathcal X$ is radial, each of its orbits is uniquely determined by its behavior in arbitrarily small neighborhoods of the singularity $p$, hence $f(\gamma_1)$ and $f(\gamma_2)$ cannot be contained in the same orbit of $\mathcal X$.

	Let $\gamma$ be a nonsingular orbit of $\mathcal Y_r$. Then $\gamma$ is homeomorphic to $\mathbb R$, and by Lemma \ref{lem:pullback-basic}(b), $f(\gamma)$ is contained in a nonsingular orbit of $\mathcal X$, which is also homeomorphic to $\mathbb R$. Again by Lemma \ref{lem:pullback-basic}(b),
	\[
	\frac{d}{dt}f(\gamma(t))=\mathcal X(f(\gamma(t)))\neq0
	\]
	for $t$ in the maximal interval of definition of $\gamma$, which implies that $f(\gamma(t))$ is strictly monotone along the orbit. In particular, $f|_\gamma$ is injective.

	Finally, let $x_1,x_2\in U$ be such that $f(x_1)=f(x_2)$. If $f(x_1)=f(x_2)=p$, then $x_1=x_2=x_0$ since $x_0$ is the unique preimage of $p$. Otherwise, let $\gamma_1$ and $\gamma_2$ be the nonsingular orbits of $\mathcal Y_r$ through $x_1$ and $x_2$. Then $f(\gamma_1)$ and $f(\gamma_2)$ intersect, hence they are contained in the same orbit of $\mathcal X$, which implies $\gamma_1=\gamma_2$ by the first part of the proof. Since $f$ is injective on each orbit, it follows that $x_1=x_2$. Therefore $f$ is injective on $U$, and being a $\mathcal C^2$ local diffeomorphism, it is a global diffeomorphism from $U$ onto $f(U)$.

\bigskip

\subsection{Proof of Theorem \ref{Th2}}
	``{\rm(i)}$\Rightarrow${\rm(ii)}''. We first note that $\nabla H=Df^Tf=0$ if and only if $f(x)=0$ if and only if $\mathcal Y_r=(Df)^{-1}f=0$, because $Df(x)$ is invertible. So the singularities of $J_f\cdot\mathcal Y_r$ and $-\nabla H$ are exactly the zeros of $f$ (i.e., $f^{-1}(0)$). Let $p\in f^{-1}(0)$ be a singularity. We claim that $p$ is a stable node for both vector fields $J_f\cdot\mathcal Y_r$ and $-\nabla H$. On one hand,
	\begin{align}
	D(J_f\mathcal Y_r)(p)
	&=\nabla J_f(p)^T\mathcal Y_r(p)+J_f(p)D\mathcal Y_r(p)\notag\\
	&=J_f(p)D\mathcal Y_r(p)\notag\\
	&=J_f(p)D\bigl((Df)^{-1}f\bigr)(p)\notag\\
	&=J_f(p)\bigl(D(Df)^{-1}(p)f(p)+(Df(p))^{-1}Df(p)\bigr)\notag\\
	&=J_f(p)I,
	\label{JfI}
	\end{align}
	where we use the facts $\mathcal Y_r(p)=f(p)=0$ and $(Df(p))^{-1}Df(p)=I$. Since $J_f(p)<0$, all eigenvalues of $D(J_f\mathcal Y_r)(p)$ are negative, so $p$ is a stable node of $J_f\mathcal Y_r$. On the other hand,
	\begin{align}
	D(-\nabla H)(p)
	&=-D\bigl(Df(p)^Tf(p)\bigr)\notag\\
	&=-\left(\sum_{i=1}^n f_i(p)D^2f_i(p)+Df(p)^TDf(p)\right)\notag\\
	&=-Df(p)^TDf(p),
	\label{gradH}
	\end{align}
	which is negative definite, so all its eigenvalues are strictly negative. Hence $p$ is also a stable node of $-\nabla H$. In particular, every singularity is isolated since $f(x)$ is a polynomial and a local diffeomorphism, so $\{x\in\mathbb R^n:f(x)=0\}$ must be finite. Write
	\[
	\{x\in\mathbb R^n:f(x)=0\}=\{p_1,\ldots,p_\ell\}.
	\]

	We next show that every positive orbit of $J_f\mathcal Y_r$ or $-\nabla H$ converges to one of these singularities. Let $\gamma_t(x_0)$ (resp. $\beta_t(x_0)$) be an orbit of $J_f\mathcal Y_r$ (resp. $-\nabla H$) with initial condition $x_0\in\mathbb R^n$. By assumption, the positive orbits of $J_f\mathcal Y_r$ (or $-\nabla H$) are bounded, hence the $\omega$-limit sets $\omega_1(x_0)$ (resp. $\omega_2(x_0)$) are nonempty, connected and invariant. Now we examine the function $H$ along orbits. Note that
	\begin{align}
	\frac{d}{dt}H(\gamma_t)
	&=\langle\nabla H(\gamma_t),J_f\mathcal Y_r(\gamma_t)\rangle\notag\\
	&=\left\langle Df(\gamma_t)^Tf(\gamma_t),J_f(\gamma_t)Df(\gamma_t)^{-1}f(\gamma_t)\right\rangle\notag\\
	&=J_f(\gamma_t)\|f(\gamma_t)\|^2\notag\\
	&=2J_f(\gamma_t)H(\gamma_t),
	\label{inequ}
	\end{align}
	where we use the facts $\langle A^Tu,v\rangle=\langle u,Av\rangle$ and $Df(\gamma_t)Df(\gamma_t)^{-1}=I$. Since $J_f(x)<0$ for all $x$, we get $\frac{d}{dt}H(\gamma_t)\leq0$, and equality holds if and only if $H(\gamma_t)=0$, or equivalently $f(\gamma_t)=0$. On the other hand, we note that
	\[
	\frac{d}{dt}H(\beta_t)=\langle\nabla H(\beta_t),\dot\beta_t\rangle=-\|\nabla H(\beta_t)\|^2\leq0,
	\]
	and equality holds if and only if $\nabla H(\beta_t)=0$, equivalently $f(\beta_t)=0$. Therefore, in both cases, $H(\gamma_t)$ and $H(\beta_t)$ are nonincreasing and bounded from below by $0$, so the limits
	\[
	c_1:=\lim_{t\to+\infty}H(\gamma_t),\qquad c_2:=\lim_{t\to+\infty}H(\beta_t)
	\]
	exist for $c_1,c_2\in\mathbb R$. Then we have $H\equiv c_1$ on $\omega_1(x_0)$ and $H\equiv c_2$ on $\omega_2(x_0)$. In what follows, we only discuss the vector field $J_f\mathcal Y_r$ and orbits $\gamma_t$; the case of $-\nabla H$ and $\beta_t$ can be treated analogously. For every $y\in\omega_1(x_0)$, invariance of $\omega_1(x_0)$ yields $\gamma_t(y)\in\omega_1(x_0)$ for $t\in\mathbb R$, so $H(\gamma_t(y))\equiv c_1$. Differentiating with respect to $t$, and by \eqref{inequ}, we conclude that $f(\gamma_t(y))=0$. In particular, $f(y)=0$, so
	\[
	\omega_1(x_0)\subset\{p_1,\ldots,p_\ell\}.
	\]
	Since $\omega_1(x_0)$ is connected and $\{p_1,\ldots,p_\ell\}$ is finite, it follows that $\omega_1(x_0)=\{p_i\}$ for some $i\in\{1,\ldots,\ell\}$. Hence every positive orbit $\gamma_t$ converges to one of the singularities.

	We claim that the nonsingular negative orbit of $\gamma_t$ of $J_f\mathcal Y_r$ with initial point $x_0$ tends to infinity as $t\to-\infty$. Otherwise, suppose that the nonsingular negative orbit of $\gamma_t$ is bounded. Then the $\alpha$-limit set $\alpha_1(x_0)$ is nonempty, connected and invariant. By \eqref{inequ} we have $\frac{d}{dt}H(\gamma_t)<0$; together with negative boundedness we obtain
	\[
	\lim_{t\to-\infty}H(\gamma_t)=c\in\mathbb R
	\]
	and so $H(\alpha_1(x_0))=c$. Same as the above process, we take $\tilde y\in\alpha_1(x_0)$ and so $H(\gamma_t(\tilde y))\equiv c$ by the invariance of $\alpha_1(x_0)$. Differentiating with respect to $t$, and by \eqref{inequ}, we obtain that $f(\gamma_t(\tilde y))=0$. Thus
	\[
	\alpha_1(x_0)\subset\{p_1,\ldots,p_\ell\}.
	\]
	However, $p_1,\ldots,p_\ell$ are all stable nodes of $J_f\mathcal Y_r$, which contradicts the definition of $\alpha$-limit set. What we claimed is true, that is every nonsingular negative orbit of $\gamma_t$ of $J_f\mathcal Y_r$ with initial point $x_0$ tends to infinity as $t\to-\infty$. This implies that the infinity of $J_f\mathcal Y_r$ is a repeller. By Lemma \ref{Lemma index}, \eqref{JfI} and \eqref{gradH}, it is easy to see that the index of the singularity $p_i$ ($i=1,\ldots,\ell$) of $J_f\mathcal Y_r$ and $-\nabla H$ is $(-1)^n$. Therefore by Theorem \ref{Th index} we have
	\[
	\sum_{j=1}^\ell i_{J_f\mathcal Y_r}(p_j)=\ell(-1)^n=(-1)^n
	\]
	and
	\[
	\sum_{j=1}^\ell i_{-\nabla H}(p_j)=\ell(-1)^n=(-1)^n.
	\]
	In both cases, we obtain $\ell=1$ and so $J_f\mathcal Y_r$ and $-\nabla H$ only have one singularity.

	Finally, since $f(0)=0$, the origin is a singularity of $J_f\mathcal Y_r$; it must be the unique singularity. Analogously, we can prove that the origin is the unique singularity of $-\nabla H$. As shown above, every positive orbit converges to some singularity, so now every positive orbit converges to the origin. Since the origin is a stable node, the origin is globally asymptotically stable.

	``{\rm(ii)}$\Rightarrow${\rm(iii)}''. If the origin of $J_f\mathcal Y_r$ is globally asymptotically stable, then $J_f\mathcal Y_r$ is topologically equivalent to the canonical radial vector field $\mathcal X_r$ defined in \eqref{canonical-radial}. It follows from Theorem \ref{Th1} that $f$ is injective. If the origin of $-\nabla H$ is globally asymptotically stable, which means that it is the unique singularity. Hence it is also the unique singularity of $J_f\mathcal Y_r$, so we conclude that the origin of $J_f\mathcal Y_r$ is globally asymptotically stable. As shown above, $f$ is injective.

	``{\rm(iii)}$\Rightarrow${\rm(i)}''. It immediately follows from Theorem \ref{Th1} that $J_f\mathcal Y_r$ is topologically equivalent to the canonical radial vector field $\mathcal X_r$. This implies that the vector field $J_f\mathcal Y_r$ is globally asymptotically stable and thus it is bounded. The global asymptotic stability of the origin of $J_f\mathcal Y_r$ also indicates that the origin of $-\nabla H$ is globally asymptotically stable, hence $-\nabla H$ is bounded.

    \bigskip

\subsection{Proof of Theorem \ref{Th3}}
Before proving Theorem \ref{Th3} we need the following result to determine the existence conditions and eigenvalues of singularities of compactified vector fields on $z_n=0$ in charts $U_k$ and $V_k$.
\begin{lemma}
	\label{lemma poincare}
	Let $\mathscr{X}$ be the polynomial vector field \eqref{poly vec} of weighted
	degree $d$ with respect to $\mathbf w$, and let
	$(P_{1,d},\dots,P_{n,d})$ be its $\mathbf w$-quasi-homogeneous part of degree
	$d$. In the chart $U_k$ a point
	$\tilde p=(z_1,\ldots,z_{n-1},0)$, represented by $\tilde z^{(k)}$ at weighted
	infinity, is a singularity of \eqref{Uk sys} if and only if
\[
P_{\sigma(i),d}(\tilde z^{(k)})
-
\frac{w_{\sigma(i)}}{w_k}
z_iP_{k,d}(\tilde z^{(k)})
=0,
\qquad i=1,\ldots,n-1.
\]
The entire hyperplane $z_n=0$ is filled with singularities if and only if these
equations hold for all $(z_1,\ldots,z_{n-1})$. In that case, the Jacobian matrix
at each point of $z_n=0$ has at least $n-1$ zero eigenvalues, and its remaining
eigenvalue is
$\lambda_k^U=-P_{k,d}(\tilde z^{(k)})/w_k$.

Similarly, in the chart $V_k$, a point
$\hat p=(z_1,\ldots,z_{n-1},0)$, represented by $\hat z^{(k)}$, is a singularity
if and only if
\[
P_{\sigma(i),d}(\hat z^{(k)})
+
\frac{w_{\sigma(i)}}{w_k}
z_iP_{k,d}(\hat z^{(k)})
=0,
\qquad i=1,\ldots,n-1.
\]
If these equations hold for all $(z_1,\ldots,z_{n-1})$, then $z_n=0$ is filled
with singularities, the Jacobian matrix has at least $n-1$ zero eigenvalues,
and its remaining eigenvalue is
$\lambda_k^V=P_{k,d}(\hat z^{(k)})/w_k$.
\end{lemma}

\begin{proof}
	We first work in the chart $U_k$. The compactified system in $U_k$ is \eqref{Uk sys}.
	Restricting this system to $z_n=0$ gives $\frac{dz_n}{d\tau}=0$
	and
	\[
	\frac{dz_i}{d\tau}
	=
	P_{\sigma(i),d}(\tilde z^{(k)})
	-
	\frac{w_{\sigma(i)}}{w_k}
	z_iP_{k,d}(\tilde z^{(k)}),
	\qquad i=1,\ldots,n-1.
	\]
	Therefore $\tilde p$ is a singularity if and only if
	\[
	P_{\sigma(i),d}(\tilde z^{(k)})
	-
	\frac{w_{\sigma(i)}}{w_k}
	z_iP_{k,d}(\tilde z^{(k)})
	=0,
	\qquad i=1,\ldots,n-1.
	\]
	
	Now assume that these equations hold for all
	$(z_1,\ldots,z_{n-1})$. Since the first $n-1$ components
	of the compactified system vanish identically on $z_n=0$, each of them
	has a factor $z_n$. Hence one can write
	\[
	\frac{dz_i}{d\tau}=z_nG_i(z),
	\qquad i=1,\ldots,n-1,
	\]
	for some polynomials $G_i$. On the other hand,
	\[
	\frac{dz_n}{d\tau}
	=
	-
	\frac1{w_k}
	z_n\widetilde P_k(z).
	\]
	Since
	\[
	\widetilde P_k(z_1,\ldots,z_{n-1},0)
	=
	P_{k,d}(\tilde z^{(k)}),
	\]
	we have, near any point
	\[
	\tilde p:=(z_1,\ldots,z_{n-1},0),
	\]
	the expansion
	\[
	\frac{dz_n}{d\tau}
	=
	-
	\frac1{w_k}
	z_nP_{k,d}(\tilde z^{(k)})
	+
	z_n^2G_n(z)
	\]
	for some polynomial $G_n$. Therefore the Jacobian matrix at $\tilde p$
takes the form
	\[
	\begin{pmatrix}
		0 & *\\
		0 & -\dfrac1{w_k}P_{k,d}(\tilde z^{(k)})
	\end{pmatrix}.
	\]
	Thus it has at least $n-1$ zero eigenvalues, and the remaining eigenvalue is
	\[
	\lambda_k^U
	=
	-\frac1{w_k}
	P_{k,d}(\tilde z^{(k)}).
	\]
	
	The proof in the chart $V_k$ is analogous.
\end{proof}

Now we are in the position to prove Theorem \ref{Th3}.

\begin{proof}[Proof of Theorem \ref{Th3}]
	We apply Lemma \ref{lemma poincare} to the polynomial vector field $J_f\mathcal Y_r$. For simplicity we write
	\[
	\mathscr Y:=J_f\mathcal Y_r=J_fDf^{-1}f.
	\]
	We first compute the $\mathbf w$-quasi-homogeneous part of $\mathscr Y$ with highest weighted degree.

	Since each $F_i$ is $\mathbf w$-quasi-homogeneous of weighted degree $m$, the weighted Euler identity gives
	\[
	DF(x)Wx=mF(x),
	\]
	where
	\[
	W=\operatorname{diag}(w_1,\ldots,w_n).
	\]
	By assumption, $J_F(x)<0$ for $x\neq0$, so $DF(x)$ is invertible for every $x\neq0$. Hence
	\[
	DF(x)^{-1}F(x)=\frac1mWx,\qquad x\neq0.
	\]
	Now
	\[
	J_F(x)=\det DF(x)
	\]
	is $\mathbf w$-quasi-homogeneous of weighted degree $nm-\sum_{j=1}^nw_j$. It follows that the $\mathbf w$-quasi-homogeneous part of $\mathscr Y$ with highest weighted degree is
	\[
	\mathscr Y_d(x):=J_F(x)DF(x)^{-1}F(x)=\frac1mJ_F(x)Wx.
	\]
	Writing
	\[
	\mathscr Y_d=(P_{1,d},\ldots,P_{n,d}),
	\]
	we have
	\[
	P_{j,d}(x)=\frac{w_j}{m}J_F(x)x_j,\qquad j=1,\ldots,n.
	\]
	Since $J_F$ has weighted degree $nm-\sum_{j=1}^nw_j$, the $j$-th component $P_{j,d}$ has weighted degree $nm-\sum_{\ell=1}^nw_\ell+w_j$. According to the discussions in Subsection 2.3, this means that the weighted degree of $\mathscr Y$ is $nm-\sum_{\ell=1}^nw_\ell+1$.

	We now study the weighted compactification of $\mathscr Y$ at infinity. In the chart $U_k$, $\tilde z^{(k)}=(z_1,\ldots,z_{k-1},1,z_k,\ldots,z_{n-1})$ and we have
	\[
	P_{\sigma(i),d}(\tilde z^{(k)})=\frac{w_{\sigma(i)}}{m}J_F(\tilde z^{(k)})z_i,
	\]
	and
	\[
	P_{k,d}(\tilde z^{(k)})=\frac{w_k}{m}J_F(\tilde z^{(k)}).
	\]
	Thus
	\begin{align*}
	P_{\sigma(i),d}(\tilde z^{(k)})-
	\frac{w_{\sigma(i)}}{w_k}z_iP_{k,d}(\tilde z^{(k)})
	&=\frac{w_{\sigma(i)}}{m}J_F(\tilde z^{(k)})z_i-
	\frac{w_{\sigma(i)}}{w_k}z_i\cdot\frac{w_k}{m}J_F(\tilde z^{(k)})=0.
	\end{align*}
	By Lemma \ref{lemma poincare} $z_n=0$ is filled with singularities in $U_k$. Moreover, the remaining eigenvalue at such a singularity is
	\[
	\lambda_k^U=-\frac1{w_k}P_{k,d}(\tilde z^{(k)})=-\frac1mJ_F(\tilde z^{(k)}).
	\]
	Since $J_F(x)<0$ for $x\neq0$, we get $\lambda_k^U>0$.

	In the chart $V_k$, $\hat z^{(k)}=(z_1,\ldots,z_{k-1},-1,z_k,\ldots,z_{n-1})$ and we have
	\[
	P_{\sigma(i),d}(\hat z^{(k)})=\frac{w_{\sigma(i)}}{m}J_F(\hat z^{(k)})z_i,
	\]
	and
	\[
	P_{k,d}(\hat z^{(k)})=-\frac{w_k}{m}J_F(\hat z^{(k)}).
	\]
	Hence
	\begin{align*}
	P_{\sigma(i),d}(\hat z^{(k)})+
	\frac{w_{\sigma(i)}}{w_k}z_iP_{k,d}(\hat z^{(k)})
	&=\frac{w_{\sigma(i)}}{m}J_F(\hat z^{(k)})z_i+
	\frac{w_{\sigma(i)}}{w_k}z_i\left(-\frac{w_k}{m}J_F(\hat z^{(k)})\right)=0.
	\end{align*}
	Therefore $z_n=0$ is also filled with singularities in $V_k$. The remaining eigenvalue is
	\[
	\lambda_k^V=\frac1{w_k}P_{k,d}(\hat z^{(k)})=-\frac1mJ_F(\hat z^{(k)}).
	\]
	Again, since $J_F(x)<0$ for $x\neq0$, we have $\lambda_k^V>0$.

	Consequently, in every local chart at weighted infinity, the compactified vector field has a positive transverse eigenvalue along the boundary $z_n=0$. Hence the weighted infinity is repelling from the finite side. Equivalently, every positive orbit of $\mathscr Y=J_f\mathcal Y_r$ is bounded.

	By Theorem B the boundedness of the polynomial vector field $J_f\mathcal Y_r$ implies that $f$ is injective. This completes the proof.
\end{proof}

\begin{remark}
We compare Theorem \ref{Th3} with the plane case in Theorem \ref{Th-L-M2}.
Indeed, when $n=2$, write
	\[
	f=(f_1,f_2):\mathbb R^2\to\mathbb R^2,
	\]
	and let $F_1$ and $F_2$ be the highest homogeneous parts of $f_1$ and $f_2$, respectively, both of degree $m$. Then
	\[
	J_f\mathcal Y_r
	=
	\binom{f_1(f_2)_y-f_2(f_1)_y}{-f_1(f_2)_x+f_2(f_1)_x}.
	\]
	Hence the highest homogeneous parts of its two components are
	\[
	F_1(F_2)_y-F_2(F_1)_y,	\qquad
	-F_1(F_2)_x+F_2(F_1)_x.
	\]
On the other hand, we have
\[
	J_F=(F_1)_x(F_2)_y-(F_1)_y(F_2)_x.
	\]
    By Euler's identities 	$x(F_1)_x+y(F_1)_y=mF_1$ and $x(F_2)_x+y(F_2)_y=mF_2$, we obtain
\begin{equation}
	\label{planar case}
	\bigl(F_1(F_2)_y-F_2(F_1)_y,\,-F_1(F_2)_x+F_2(F_1)_x\bigr)
	= \left(\frac{x}{m}J_F,\frac{y}{m}J_F\right).
\end{equation}
Equation \eqref{planar case} shows that the condition
$J_F(x)<0$ for $x\in\mathbb{R}^n\setminus\{0\}$ in Theorem \ref{Th3} covers that in Theorem \ref{Th-L-M2}.
\end{remark}

\bigskip
\section{Discussion}
\label{S4}
\setcounter{equation}{0}

Note that Theorem \ref{Th BL} contains three planar cases: annular, radial and strip. In this
paper we have extended the radial case to arbitrary dimension. However, the other two cases do not admit such a direct extension.
We first discuss the strip case. For planar systems the canonical strip vector field
$(1,0)$ has one first integral, namely the second coordinate. Thus its orbits
are exactly the connected components of the level sets of this first integral.
In $\mathbb{R}^n$ the canonical strip vector field can be written as $\mathcal{X}_s=(1,0,\ldots,0)$,
which has $n-1$ independent first integrals $x_2,\ldots,x_n$. If $\mathcal{X}$ is
topologically equivalent to $\mathcal{X}_s$ by a homeomorphism
$T=(T_1,\ldots,T_n)$, then $T_2,\ldots,T_n$ are continuous first integrals of
$\mathcal{X}$. Therefore by Lemma \ref{lem:pullback-basic} (c), for the pullback vector field $\mathcal{Y}_\mathcal{X}$, the functions
\[
T_2\circ f,\ldots,T_n\circ f
\]
are first integrals. Nevertheless, in dimension $n\geq 3$, the common level set
of these $n-1$ first integrals may have several connected components. Hence a
level set need not be a single orbit of $\mathcal{Y}_\mathcal{X}$. This is the main obstruction to
a direct higher-dimensional version of the strip case. To recover the planar result, one has to impose an additional connectedness assumption on the
common level sets.

\begin{example}
 Consider the local
	diffeomorphism
	\[
	f:\mathbb{R}^3\rightarrow \mathbb{R}\times
	(\mathbb{R}^2\setminus\{0\}),\qquad
	f(x,y,z)=\bigl(x,e^y\cos z,e^y\sin z\bigr).
	\]
	Its Jacobian determinant is
	\[
	J_f(x,y,z)=e^{2y}\neq 0.
	\]
	Take the canonical strip vector field $\mathcal{X}_s=(1,0,0)$ on the image of $f$. Since
	\[
	Df(x,y,z)(1,0,0)^T=(1,0,0)^T,
	\]
	the pullback of $\mathcal{X}_s$ by $f$ is 
	\[
	\mathcal{Y}_s:=Df^{-1}\mathcal{X}_s=(1,0,0).
	\]
	Thus the pullback vector field $\mathcal{Y}_s$ has the same coordinate
	expression as $\mathcal{X}_s$. However,
	$f$ is not injective, because
	\[
	f(x,y,z)=f(x,y,z+2\pi)
	\]
	for all $(x,y,z)\in\mathbb{R}^3$. The common level sets of the first integrals
	\[
	f_2=e^y\cos z,\qquad f_3=e^y\sin z
	\]
	are not connected. Indeed, if $c=(c_2,c_3)\neq(0,0)$ and
	$c_2=e^{y_0}\cos z_0$, $c_3=e^{y_0}\sin z_0$, then
	\[
	\{(x,y,z): f_2(x,y,z)=c_2,\ f_3(x,y,z)=c_3\}
	=
	\bigcup_{k\in\mathbb{Z}}
	\{(x,y_0,z_0+2k\pi):x\in\mathbb{R}\}.
	\]
	This is a disjoint union of strip orbits. The example illustrates how
	disconnected common level sets can arise and why a connectedness hypothesis is
	natural. It is not, by itself, a counterexample to Theorem \ref{Th4} without
	the connectedness hypothesis: the ambient spaces $\mathbb R^3$ and
	$\mathbb R\times(\mathbb R^2\setminus\{0\})$ are not homeomorphic, so the two
	strip fields in this example are not topologically equivalent in the sense
	used here.
	\end{example}

The annular case is even more delicate. For planar systems an annular vector field
can be written as $(-x_2,x_1)$ on $\mathbb{R}^2\setminus\{0\}$,
and its orbits are closed curves surrounding the origin. This description
relies essentially on planar topology, especially the Jordan curve theorem:
every simple closed curve separates the plane into a bounded and an unbounded
component. In higher dimensions a closed orbit is still one-dimensional, but
it no longer separates the ambient space. 

Now we extend the strip case of Theorem \ref{Th BL} to $\mathbb{R}^n$.
We denote the {\it canonical strip vector field} in $\mathbb{R}^n$ by
\begin{equation}\label{canonical-strip}
	\mathcal{X}_s(x)=e_1=(1,0,\dots,0),\qquad x\in\mathbb{R}^n.
\end{equation}
A vector field $\mathcal{X}:U\to\mathbb{R}^n$ is said to be a  \textit{strip vector field} if
$\mathcal{X}$  is  topologically equivalent  to
$\mathcal{X}_s$.
The pullback of the canonical strip vector field \eqref{canonical-strip} by $f$ takes the form
\begin{equation}\label{eq:strip-pullback}
	\mathcal{Y}_s(x)=Df(x)^{-1}e_1,\qquad x\in U.
\end{equation}
That is, $\mathcal{Y}_s$ is the  vector field satisfying
\begin{equation}\label{strip-system}
	Df(x)\mathcal{Y}_s(x)=e_1.
\end{equation}
Then \eqref{strip-system} is equivalent to
\begin{equation}
	\label{strip-system-components}
	\langle \nabla f_1(x),\mathcal{Y}_s(x)\rangle =1,\qquad
	\langle \nabla f_j(x),\mathcal{Y}_s(x)\rangle =0,\quad j=2,\dots,n.
\end{equation}
Hence $f_2,\dots,f_n$ are first integrals of $\mathcal{Y}_s$, while $f_1$
is strictly increasing along every orbit of $\mathcal{Y}_s$.
Moreover, by \eqref{strip-system-components} we can also write
\begin{equation*}
	\mathcal{Y}_s = \frac{1}{J_f(x)} \left( \nabla f_2 \times \nabla f_3 \times \cdots \times \nabla f_n \right),
\end{equation*}
where $\times$ denotes the generalized cross product in $\mathbb R^n$, with
the convention
\[
\langle v,\nabla f_2\times\cdots\times\nabla f_n\rangle
=\det(v,\nabla f_2,\ldots,\nabla f_n)
\quad\text{for every }v\in\mathbb R^n.
\]
Note that the canonical strip vector field $\mathcal{X}_s$ given in
\eqref{canonical-strip} has functionally independent first integrals $x_j$ for
$j=2,\ldots,n$. Let $\mathcal{X}$ be a strip vector field, and let
$T=(T_1,\ldots,T_n)$ give a topological equivalence from $\mathcal{X}$ to
$\mathcal{X}_s$. Since $T$ maps orbits of $\mathcal X$ to orbits of
$\mathcal X_s$, the continuous functions
\begin{equation*}
	x_j \circ T=T_j,\quad j=2,\ldots,n
\end{equation*}
are first integrals of $\mathcal X$.
\begin{mytheorem}
	\label{Th4}
	Let $U$ be a simply connected open subset of $\mathbb{R}^n$. Let
	$f=(f_1,\dots,f_n):U\to\mathbb{R}^n$ be a $\mathcal{C}^2$ local
	diffeomorphism and let $\mathcal{X}:f(U)\to\mathbb{R}^n$ be a
	$\mathcal{C}^1$ strip vector field. Let
	$T=(T_1,\ldots,T_n):f(U)\to\mathbb R^n$ be a
	homeomorphism that gives a topological equivalence between $\mathcal X$ and
	$\mathcal X_s$. Moreover, suppose that, for every
	$c=(c_2,\dots,c_n)\in\mathbb{R}^{n-1}$, the set
	\begin{equation}
		\label{fiber F}
		\mathscr{F}_c:=\{x\in U:\ T_2\circ f=c_2,\dots,T_n \circ f=c_n\}
	\end{equation}
	is either empty or connected. Then the
	following two conditions are equivalent:
	\begin{enumerate}[leftmargin=*]
		\item[\rm(i)] The vector field $\mathcal{Y}_\mathcal{X}$ is topologically equivalent to the vector field $\mathcal{X}$.
		\item[\rm(ii)] The map $f$ is a global diffeomorphism onto $f(U)$.
	\end{enumerate}
\end{mytheorem}
\begin{proof}
	The implication {\rm(ii)}$\Rightarrow${\rm(i)} is immediate because $f$
	gives the required orbit equivalence. We prove the converse. Let $\gamma_1$
	and $\gamma_2$ be two distinct orbits of
	$\mathcal{Y}_{\mathcal X}$. We claim that $f(\gamma_1)$ and $f(\gamma_2)$ are
	contained in distinct orbits of $\mathcal X$.
	Set
	\[
	H:=(T_2,\dots,T_n)\circ f:U\to\mathbb{R}^{n-1}.
	\]
	By statement (c) of Lemma \ref{lem:pullback-basic}, it follows that $H$
	is constant along the orbits of $\mathcal{Y}_{\mathcal X}$.
	Fix $c\in\mathbb{R}^{n-1}$. Then $\mathscr{F}_c=H^{-1}(c)$ is, by assumption,
	either empty or connected. On the other hand, $\mathscr{F}_c$ is a union of
	orbits of $\mathcal{Y}_{\mathcal X}$, because $H$ is constant on each orbit.
	Moreover, since $T\circ f$ is a local homeomorphism, every point of
	$\mathscr{F}_c$ has a neighborhood in which $\mathscr{F}_c$ coincides with a
	single local orbit of $\mathcal{Y}_{\mathcal X}$. Hence every orbit contained
	in $\mathscr{F}_c$ is open in $\mathscr{F}_c$. These orbits partition the
	fiber, so the connectedness of every nonempty $\mathscr{F}_c$ implies that the
	fiber consists of exactly one orbit.
	
	Now, if $f(\gamma_1)$ and $f(\gamma_2)$ were contained in the same orbit of
	$\mathcal X$, then $T\circ f(\gamma_1)$ and $T\circ f(\gamma_2)$ would be
	contained in the same orbit of $\mathcal X_s$, so that
	\[
	H(\gamma_1)=H(\gamma_2)=c
	\]
	for some $c\in\mathbb{R}^{n-1}$. Therefore
	\[
	\gamma_1\cup\gamma_2\subset \mathscr{F}_c,
	\]
	which is a contradiction. This proves the claim.
	
	Next we prove that $f$ is injective on each orbit of $\mathcal{Y}_{\mathcal X}$.
	Since $\mathcal{Y}_{\mathcal X}$ is topologically equivalent to $\mathcal X$ and
	$\mathcal X$ is strip, all its orbits are homeomorphic to $\mathbb{R}$. By statement (b) of Lemma \ref{lem:pullback-basic}, if $\gamma$ is an orbit of $\mathcal{Y}_{\mathcal X}$, then
	$f(\gamma)$ is contained in an orbit of $\mathcal X$, which is also homeomorphic
	to $\mathbb{R}$. If $\gamma$ is parametrized by flow time, then
	$f\circ\gamma$ is an integral curve of $\mathcal X$. Since strip orbits are
	nonperiodic, uniqueness of solutions shows that the restriction
	\[
	f|_{\gamma}:\gamma\to f(\gamma)
	\]
	is injective.
	
	Let $x_1,x_2\in U$ satisfy $f(x_1)=f(x_2)$. The images of the orbits through
	$x_1$ and $x_2$ then meet and hence lie in the same orbit of $\mathcal X$.
	By the claim, the two points belong to the same orbit of
	$\mathcal Y_{\mathcal X}$; the injectivity of $f$ on that orbit gives
	$x_1=x_2$. Thus $f$ is injective. Since $f$ is a local diffeomorphism, it is a
	global diffeomorphism from $U$ onto $f(U)$.
\end{proof}

\bigskip
\section*{Acknowledgements}
The first and third authors are supported by the National Key R\&D Program of China (No. 2022YFA1005900) and 
the National Natural Science Foundations of China (Nos. 12671199).
The second author is partially supported by the Agencia Estatal de
Investigaci\'on of Spain grant PID2022-136613NB-I00.

\end{document}